\documentclass[a4paper]{amsart}
\usepackage{hyperref}
\usepackage[top=1in, bottom=2in, left=1.25in, right=1.25in]{geometry}
\usepackage{amsmath,amssymb,amsthm,amsfonts}
\usepackage{pst-node}
\usepackage{tikz-cd} 
\theoremstyle{definition}
\newtheorem{defn}{Definition}[section]

\newtheorem{prop}[defn]{Proposition}

\newtheorem{thm}[defn]{Theorem}
\newtheorem{corr}[defn]{Corollary}
\newtheorem{lem}[defn]{Lemma}
\newtheorem{obs}[defn]{Observation}
\newtheorem{question}[defn]{Question}
\newtheorem{remark}[defn]{Remark}
\newtheorem{claim}[defn]{claim}

\title{On some applications of strongly compact Prikry forcing}
\author{Amitayu Banerjee}

\address{Department of Logic, Institute of Philosophy, E\"otv\"os Lor\'and University,
M\'{u}zeum krt. 4/i Budapest, H-1088 Hungary}
\email{banerjee.amitayu@gmail.com}
\keywords{strongly compact Prikry forcing, symmetric inner model of a forcing extension, appropriate automorphism technique}
\begin{document}
\maketitle
\begin{abstract} We work with symmetric inner models of forcing extensions based on {\em strongly compact Prikry forcing} to extend some known results.
\end{abstract}

\section{Introduction}
\subsection{Extending a result of Dimitriou I} Apter, Dimitriou and Koepke \cite{ADK2016} proved that in Gitik's model \cite{Git1980}, every singular cardinal is a Rowbottom cardinal with a Rowbottom filter.
Further in \cite{ADK2016}, they conjectured about the possibility of removing the additional assumption that {\em `every strongly compact cardinal is a limit of measurable cardinals'}. Apter communicated to us that the methods of \cite{AH1991} can be applied to prove the conjecture. The conjecture is still open, but inspired from the {\em appropriate automorphism technique} used in [\cite{AH1991}, \textbf{Lemma 3.1}], we construct a model based on {\em strongly compact prikry forcing}, with a sequence of successive singular Rowbottom cardinals that has order type larger than $\omega$, and smaller than or equal to $(\omega_{1})^{V}$, if $V$ is the ground model. This may remove the additional assumption that {\em `every strongly compact cardinal is a limit of measurable cardinals'} from [\cite{Dim2011}, \textbf{corollary 2.32}].

\begin{thm}
{\em Suppose for some ordinal $\rho\in (\omega,\omega_{1}]$, there is a $\rho$-long sequence $\langle \kappa_{\epsilon}:0<\epsilon<\rho\rangle$ of strongly compact cardinals, which sequence has limit $\eta$ in a ground model $V$ of ZFC. Then there is a forcing extension $V[G]$ that has a symmetric inner model 
$\mathcal{N}$ in which the following hold.
\begin{enumerate}
    \item All cardinals in the interval $(\omega,\eta)$ are uncountable and singular of cofinality $\omega$. 
    \item All cardinals in the interval $(\omega,\eta)$ carry a Rowbottom filter and almost Ramsey.
\end{enumerate}
}\end{thm}

\subsection{Extending a result of Dimitriou II} Inspired by a question of L\"{o}we, Dimitriou constructed a symmetric extension in [\cite{Dim2011}, \textbf{Chapter 2}, \textbf{$\S$3}] with a countable sequence of any desired pattern of regular cardinals and singular cardinals of cofinality $\omega$. We observe that if we replace the {\em injective tree Prikry forcing} by {\em strongly compact Prikry forcing} in the proof of [\cite{Dim2011}, \textbf{Theorem 2.12}] then we can obtain a countable sequence of any desired pattern of regular cardinals and singular Rowbottom cardinals of cofinality $\omega$. Here we also apply the appropriate automorphism technique.

\begin{thm}{\em Suppose there is an increasing sequence $\langle \kappa_{n}:0<n<\omega\rangle$ of strongly compact cardinals in a ground model $V$ of ZFC, which sequence has limit $\eta$. Then for any function $f:\omega\rightarrow 2$ in the ground model, there is a forcing extension $V[G]$ that has a symmetric inner model 
$\mathcal{N}$ in which the following hold.
\begin{enumerate}
    \item $\aleph_{n+1}$ is regular if $f(n)=1$ and singular if $f(n)=0$.
    \item Each singular cardinal in the obtained pattern of regular and singular cardinals, are almost Ramsey and carry a Rowbottom filter.
    \item Each regular cardinal in the obtained pattern of regular and singular cardinals, do not carry any unifrom ultrafilter.
\end{enumerate}
}
\end{thm}

\subsection{Reducing the assumption of supercompactness by strong compactness} Apter and Cody (c.f. [\cite{AC2013}, \textbf{Theorem 1}]) obtained a symmetric inner model of a forcing extension where $\kappa$ and $\kappa^{+}$ are both singular, and there is a sequence of distinct subsets of $\kappa$ of length equal to any predefined ordinal, assuming a supercompact cardinal $\kappa$. They used the fact that it is possible to obtain a forcing extension where a supercompact cardinal $\kappa$ can become indestructible under $\kappa$-directed closed forcing notions\footnote{Using Laver's indestructibility of supercompactness.} and worked on a symmetric inner model of a forcing extension based on {\em supercompact Prikry forcing} to obtain the result. We observe that applying a recent result of Usuba (c.f. [\cite{ADU2019}, \textbf{Theorem 3.1}]), followed by working on a symmetric inner model of a forcing extension based on strongly compact Prikry forcing, it is possible to weaken the assumption of a supercompact cardinal $\kappa$ to a strongly compact cardinal $\kappa$.

\begin{obs}
{\em Suppose $\kappa$ is a strongly compact cardinal, GCH holds, and $\theta$ is an ordinal in a ground model $V$ of ZFC. Then there is a forcing extension $V[G]$ that has a symmetric inner model $\mathcal{N}$ in which the following hold.
\begin{enumerate}
    \item $\kappa$ and $\kappa^{+}$ are both singular with $(cf(\kappa))^{\mathcal{N}}=\omega$ and $(cf(\kappa^{+}))^{\mathcal{N}}<\kappa$.
    \item $\kappa$ is a strong limit cardinal that is a limit of inaccessible cardinals.
    \item There is a sequence of distinct subsets of $\kappa$ of length $\theta$.
    \item $AC_{\kappa}$ fails.
\end{enumerate}
}
\end{obs}

\subsection{Strong compactness and normal measures} Apter \cite{Apt2006} proved that $\aleph_{\omega+1}$ can carry $\geq\aleph_{\omega+2}$ number of normal measures in ZF. The methods of \cite{Apt2006} easily generalizes to handling successors of other singular cardinals of cofinality $\omega$. Thus it is known that a successor $\kappa^{+}$ of a singular cardinal $\kappa$ of cofinality $\omega$ can carry $\geq \kappa^{++}$ number of normal measures in ZF.
Recently, Goldberg \cite{Gol2018} introduced the Ultrapower axiom (UA). 
Assuming UA, Apter proved that if $\lambda$ is a measurable cardinal such that the order of $\lambda$ is $\delta$, then
the number of normal measures $\lambda$ carries is $\vert\delta\vert$ (c.f. [\cite{Apt2020}, \textbf{Proposition 1}]). Applying this result we can construct a symmetric inner model based on strongly compact Prikry forcing where the successor of a singular cardinal of cofinality $\omega$, can carry arbitrary (regular cardinal) number of normal measures under certain large cardinal assumptions. 

\begin{remark}{\em Let $V$ be a model of ZFC + GCH+ UA. In $V$, let $\kappa<\lambda$ are such that $\kappa$ is strongly compact and $\lambda$ is the least measurable cardinal above $\kappa$ such that $o(\lambda)=\delta$ for some ordinal $\delta\leq \lambda^{++}$. Then there is a forcing extension $V[G]$ that has a symmetric inner model $\mathcal{N}$  where $cf(\kappa)=\omega$ and $\kappa^{+}$ can carry $\vert\delta\vert^{\mathcal{N}}$ number of normal measures.}\end{remark}

It follows from the above Remark that if $1 \leq \delta < \omega$, then a successor of a singular cardinal of countable cofinality will carry a precise finite number of normal measures (e.g., 3, 97, 4962, etc.) in $\mathcal{N}$. Further, if $\delta = \omega, \omega_{3}$, or $\omega_{297}$, then a successor of a singular cardinal of countable cofinality will carry exactly $\aleph_{0}$, $\aleph_{3}$, or $\aleph_{297}$ normal measures respectively in $\mathcal{N}$. 
\section{Basics}
\subsection{Large Cardinals}
We recall the definition of inaccessible cardinals in the context of ZFC and other large cardinals in the context of ZF. In ZFC, we say $\kappa$ is a strongly inaccessible cardinal if it is a regular strong limit cardinal where the definition of ``strong limit" is that for all $\alpha<\kappa$, we have $2^{\alpha}<\kappa$. 
We recall the other necessary large cardinal definitions in the context of ZF from {\em `The Higher Infinite'} \cite{Kan2003} of Akihiro Kanamori.

\begin{defn} Given an uncountable cardinal $\kappa$, we recall the following definitions.
\begin{enumerate}
\item $\kappa$ is almost Ramsey if for all $\alpha<\kappa$ and $f:[\kappa]^{<\omega}\rightarrow 2$, there is a homogeneous set $X\subseteq \kappa$ for $f$ having order type $\alpha$.

\item  $\kappa$ is $\mu$-Rowbottom if for all $\alpha<\kappa$ and $f:[\kappa]^{<\omega}\rightarrow \alpha$, there is a homogeneous set $X\subseteq\kappa$ for $f$ of order type $\kappa$ such that $\vert f^{''}[X]^{<\omega}\vert<\mu$. $\kappa$ is Rowbottom if it is $\omega_{1}$-Rowbottom. A filter $\mathcal{F}$ on $\kappa$ is a Rowbottom filter on $\kappa$ if for any $f:[\kappa]^{<\omega}\rightarrow\lambda$, where $\lambda<\kappa$ there is a set $X\in\mathcal{F}$ such that $\vert f^{''}[X]^{<\omega}\vert\leq\omega$.

\item $\kappa$ is measurable if there is a $\kappa$-complete free ultrafilter on $\kappa$. 
A filter $\mathcal{F}$ on a cardinal $\kappa$ is normal if it is closed under diagonal intersections:

\begin{center}
    If $X_{\alpha}\in \mathcal{F}$ for all $\alpha<\kappa$, then $\Delta_{\alpha<\kappa}X_{\alpha}\in \mathcal{F}$. 
\end{center}

In ZF we have the following lemma. 

\begin{lem}{(c.f. [\cite{Dim2011}, \textbf{Lemma 0.8}]).}
{\em An ultrafilter $\mathcal{U}$ over $\kappa$ is normal if and only if for every regressive $f:\kappa\rightarrow\kappa$ there is an $X\in \mathcal{U}$ such that $f$ is constant on $X$.} 
\end{lem}

Thus, we say an ultrafilter $\mathcal{U}$ over $\kappa$ is normal if for every regressive $f:\kappa\rightarrow\kappa$ there is an $X\in \mathcal{U}$ such that $f$ is constant on $X$.

\item For a set $A$ we say $\mathcal{U}$ a fine measure on $\mathcal{P}_{\kappa}(A)$ if $\mathcal{U}$ is a $\kappa$-complete ultrafilter and for any $i\in A$, $\{x\in\mathcal{P}_{\kappa}(A): i\in x\}\in\mathcal{U}$. We say $\mathcal{U}$ is a normal measure on $\mathcal{P}_{\kappa}(A)$, if $\mathcal{U}$ is a fine measure and if $f:\mathcal{P}_{\kappa}(A)\rightarrow A$ is such that $f(X)\in X$ for a set in $\mathcal{U}$, then $f$ is constant on a set in $\mathcal{U}$. $\kappa$ is $\lambda$-strongly compact if there is a fine measure on $ \mathcal{P}_{\kappa}(\lambda)$, it is strongly compact if it is $\lambda$-strongly compact for all $\kappa\leq\lambda$.

\item $\kappa$ is $\lambda$-supercompact if there is a normal measure on $ \mathcal{P}_{\kappa}(\lambda)$, it is supercompact if it is $\lambda$-supercompact for all $\kappa\leq\lambda$.
\end{enumerate}
\end{defn}

\textbf{Remark 1.} We note that the definition of supercompact (similarly strongly compact) is meant in the terms of ultrafilters, which is weaker than the definition of supercompact in terms of elementary embedding due to Woodin \textbf{[\cite{Wood2010}, Definition 220]} (e.g. $\aleph_{1}$ can be supercompact or strongly compact if we consider the definition of supercompact or strongly compact in terms of ultrafilters, but $\aleph_{1}$ can not be the critical
point of an elementary embedding). 

\textbf{Remark 2.} Ikegami and Trang [$\S$2, \cite{IT2019}] defined that an ultrafilter $\mathcal{U}$ on $\mathcal{P}_{\kappa}(X)$ is normal if for any set $A \in \mathcal{U}$ and $f : A \rightarrow \mathcal{P}_{\kappa}(X)$
with $\emptyset \not= f(\sigma) \subseteq \sigma$ for all $\sigma\in A$, there is an $x_{0} \in X$ such that for $\mathcal{U}$-measure one many $\sigma$ in $A$, $x_{0}\in f(\sigma)$. They note that their definition of normality is equivalent to the closure under diagonal intersections in ZF, while it may not be equivalent to the definition of normality in our sense without AC.

From now on, all our inaccessible cardinals are strongly inaccessible.


\subsection{Homogeneity of forcing notions} We recall the definition of {\em weakly homogeneous} and {\em cone homogeneous} forcing notions from \cite{DF2008} (c.f. [\cite{DF2008}, \textbf{Definition 2}]).

\begin{defn}
{\em Let $\mathbb{P}$ be a set forcing notion.
\begin{itemize}
    \item We say $\mathbb{P}$ is weakly homogeneous if for any $p,q\in \mathbb{P}$, there is an automorphism $a:\mathbb{P}\rightarrow\mathbb{P}$ such that $a(p)$ and $q$ are compatible.
    \item For $p\in\mathbb{P}$, let $Cone(p)$ denote $\{r\in\mathbb{P}:r\leq p\}$. We say $\mathbb{P}$ is cone homogeneous if and only if for any $p, q\in \mathbb{P}$, there exist $p'\leq p$, $q'\leq q,$ and an isomorphism $\pi: Cone(p')\rightarrow Cone(q')$.
\end{itemize}
}
\end{defn}

Following [\textbf{\cite{DF2008}, Fact 1}], if $\mathbb{P}$ is a weakly homogeneous forcing notion, then it is cone homogeneous too. Also, the finite support products of weakly (cone) homogeneous forcing notions are weakly (cone) homogeneous. A crucial feature of symmetric inner models of forcing extensions based on weakly (cone) homogeneous forcings are that they can be approximated by certain intermediate submodel where AC holds (c.f. [\textbf{\cite{Dim2011}, Lemma 1.29}]).

\subsection{Strongly compact Prikry forcing}
Suppose $\lambda>\kappa$ and $\kappa$ be a $\lambda$-strongly compact cardinal in the ground model $V$. Let $\mathcal{U}$ be a $\kappa$-complete fine ultrafilter over $\mathcal{P}_{\kappa}(\lambda)$. 

\begin{defn}{(c.f. [\cite{Git2010}, \textbf{Definition 1.51}]).}
{\em A set $T$ is called a $\mathcal{U}$-tree with trunk t if and only if the following holds.
\begin{enumerate}
    \item $T$ consists of finite sequences $\langle P_{1},...,P_{n}\rangle$ of elements of $\mathcal{P}_{\kappa}(\lambda)$ so that $P_{1}\subseteq P_{2}\subseteq ... P_{n}$.
    \item $\langle T, \unlhd\rangle$ is a tree, where $\unlhd$ is the order of the end extension of finite sequences.
    \item t is a trunk of $T$, i.e., $t\in T$ and for every $\eta\in T$, $\eta \unlhd t$ or $t \unlhd \eta$.
    \item For every $t\unlhd \eta$, $Suc_{T} (\eta)=\{Q\in \mathcal{P}_{\kappa}(\lambda): \eta \frown\langle Q\rangle\in T\}\in \mathcal{U}$.
\end{enumerate}
}
\end{defn}
The set $\mathbb{P}_{\mathcal{U}}$ consists of all pairs $\langle t,T\rangle$ such that $T$ is a
$\mathcal{U}$-tree with trunk $t$. If $\langle t,T\rangle, \langle s,S\rangle\in\mathbb{P}_{\mathcal{U}}$, we say that $\langle t,T\rangle$ is stronger than
$\langle s,S\rangle$, and denote this by $\langle t,T\rangle\geq\langle s,S\rangle$, if and only if $T \subseteq S$. We call $\mathbb{P}_{\mathcal{U}}$ with the ordering defined above as strongly compact Prikry forcing with respect to $\mathcal{U}$.
Let $G$ be $V$-generic over $\mathbb{P}_{\mathcal{U}}$.\footnote{Alternatively, we also recall the definition of a strongly compact Prikry forcing $\mathbb{P}_{\mathcal{U}}$ from \cite{AH1991}. Let $\mathcal{U}$ be a fine measure on $\mathcal{P}_{\kappa}(\lambda)$ and $\mathcal{F}=\{f: f$ is a function from $[\mathcal{P}_{\kappa}(\lambda)]^{<\omega}$ to $\mathcal{U}\}$.  In particular, $\mathbb{P}_{\mathcal{U}}$ is the set of all finite sequences of the form $\langle p_{1},...p_{n},f\rangle$ satisfying the following properties.

\begin{itemize}
    \item $\langle p_{1},...p_{n}\rangle\in [\mathcal{P}_{\kappa}(\lambda)]^{<\omega}$.
    \item for $0\leq i<j\leq n$, $p_{i}\cap \kappa\not= p_{j}\cap \kappa$.
    \item $f\in \mathcal{F}$.
\end{itemize}

The ordering on $\mathbb{P}_{\mathcal{U}}$ is given by $\langle q_{1},...q_{m},g\rangle \leq \langle p_{1},...,p_{n},f\rangle$ if and only if we have the following.

\begin{itemize}
    \item $n\leq m$.
    \item $\langle p_{1},...,p_{n}\rangle$ is the initial segment of $\langle q_{1},...,q_{m}\rangle$.
    \item For $i=n+1,..., m$, $q_{i}\in f(\langle p_{1},...,p_{n}, q_{n+1},...,q_{i-1}\rangle)$.
    \item For $\overrightarrow{s}\in [\mathcal{P}_{\kappa}(\lambda)]^{<\omega}$, $g(\overrightarrow{s})\subseteq f(\overrightarrow{s})$.
\end{itemize}

For any regular $\delta\in [\kappa,\lambda]$, we denote $r\restriction{\delta}=\{\langle p_{0}\cap \delta,...p_{n}\cap\delta\rangle: \exists f\in \mathcal{F} \left[\langle p_{0},...p_{n},f\rangle\in G\right]\}$. In $V[r\restriction \kappa]\subseteq V[G]$, $\kappa$ is a singular cardinal having cofinality $\omega$. Since any two conditions having the same stems are compatible, i.e. any two conditions of the form $\langle p_{1},...,p_{n},f\rangle$ and $\langle p_{1},...,p_{n},g \rangle$ are compatible., \textbf{$\mathbb{P}_{\mathcal{U}}$ is $(\lambda^{<\kappa})^{+}$-c.c.}
}
Following a Prikry like lemma (c.f. [\cite{Git2010}, \textbf{Theorem 1.52}] $\&$ [\cite{AH1991}, \textbf{Lemma 1.1}]), $\mathbb{P}_{\mathcal{U}}$ does not add bounded subsets to $\kappa$.
Also, $(\lambda)^{V}$ is collapsed to $\kappa$ in $V[G]$. Again, $\mathbb{P}_{\mathcal{U}}$ is $(\lambda^{<\kappa})^{+}$-c.c. Let $\delta\in[\kappa,\lambda)$ be an inaccessible cardinal. If $x\subseteq\mathcal{P}_{\kappa}(\lambda)$, let $x\restriction\delta=\{Z\cap\delta:Z\in x\}$ and $\mathcal{U}\restriction\delta=\{x\restriction\delta: x\in\mathcal{U}\}$. Since, $\mathcal{U}$ is a $\kappa$-complete, fine ultrafilter on $\mathcal{P}_{\kappa}(\lambda)$, $\mathcal{U}\restriction\delta$ is a $\kappa$-complete, fine ultrafilter on $\mathcal{P}_{\kappa}(\delta)$. Consequently, we can consider the strongly compact Prikry forcing $\mathbb{P}_{\mathcal{U}\restriction\delta}$ like $\mathbb{P}_{\mathcal{U}}$.
\subsection{Injective tree-Prikry forcing} We recall the basics of injective tree-Prikry forcing from  [\cite{Dim2011}, \textbf{Chapter 2}, $\S$1]. Let $\kappa$ be a measurable cardinal, let $\alpha \geq \kappa$ be a regular cardinal, and $\phi$ a uniform $\kappa$-complete ultrafilter over $\alpha$. 

\begin{defn}{(c.f. [\cite{Dim2011}, \textbf{Definition 2.1}]).}
{\em A set $T \subseteq ^{<\omega}\alpha$ is called an injective $\phi$-tree if and only if the following hold.
\begin{enumerate}
\item $T$ consists of finite injective sequences of elements of $\alpha$,
\item $T$ is a tree with respect to end extension ``$\unlhd$'',
\item $T$ has a trunk, i.e., an element denoted by $tr_{T}$ , that is maximal in $T$ such that for every $t \in T$, $t \unlhd tr_{T}$ or $tr_{T} \unlhd t$, and
\item for every $t \in T$ with $tr_{T}\unlhd t$,
the set $Suc_{T}(t)= \{\beta \in \alpha: t\frown \langle\beta\rangle\in T\}\in \phi$.
\end{enumerate}}
\end{defn}
The set $P^{t}_{\phi}$ consists of all injective $\phi$-trees, and it is ordered by direct inclusion, i.e., $T \leq S$ if and only if $T \subseteq S$. We call the set $P^{t}_{\phi}$ together with the ordering defined above as the injective tree-Prikry forcing with respect to the ultrafilter $\phi$. Let $G$ be a $P^{t}_{\phi}$-generic filter over $V$. In $V[G]$, the cardinals between $\kappa$ and $\alpha^{+}$ collapse (c.f. \cite{Dim2011}).

\begin{lem}{(c.f. [\cite{Dim2011}, \textbf{Lemma 2.2}]).} {\em $P^{t}_{\phi}$ does not add bounded subsets to $\kappa$ and has the $\alpha^{+}-c.c$.}
\end{lem}

\begin{lem}{(c.f. [\cite{Dim2011}, \textbf{Lemma 2.3}]).}
{\em In $V[G]$, $cf(\alpha)=\omega$.}
\end{lem}

\subsection{Mitchell order and Ultrapower Axiom}
Let $\kappa$ be a measurable cardinal, and $\mathcal{U}_{1}$ and $\mathcal{U}_{2}$ be the normal
measures on $\kappa$. Define the relation $\unlhd$ as follows.
\begin{center}
    $\mathcal{U}_{1} \unlhd \mathcal{U}_{2}$ if and only if $\mathcal{U}_{1} \in Ult_{\mathcal{U}_{2}} V$
\end{center}
The relation $\mathcal{U}_{1} \unlhd \mathcal{U}_{2}$ is the Mitchell order. The Mitchell order is well-founded (c.f. [\cite{Jec2003}, \textbf{Lemma 19.32}]). The order of $\kappa$, denoted by $o(\kappa)$, is the height of $\unlhd$. 
The Ultrapower Axiom UA, introduced by Goldberg in [\cite{Gol2018}, Definitions 2.1 – 2.3], states
the following.

\begin{defn}{\em \textbf{(Ultrapower axiom (c.f. [\cite{Gol2018}, Definitions 2.1 – 2.3])).}}
{\em Let $V$ be a model of ZFC and $\mathcal{U}_{0}$, $\mathcal{U}_{1}$ in $V$ are countably complete ultrafilters over $x_{0} \in V$, $x_{1} \in V$
respectively with $j_{\mathcal{U}_{0}}: V \rightarrow M_{\mathcal{U}_{0}}$ and $j_{\mathcal{U}_{1}}: V \rightarrow M_{\mathcal{U}_{1}}$ the associated elementary embeddings. Then there exist $W_{0}\in M_{\mathcal{U}_{0}}$ a countably complete ultrafilter over $y_{0} \in M_{\mathcal{U}_{0}}$ and $W_{1} \in M_{\mathcal{U}_{1}}$ a countably complete ultrafilter over $y_{1} \in M_{\mathcal{U}_{1}}$ such that the following hold.
\begin{enumerate}
\item  For $j_{\mathcal{W}_{0}}: M_{\mathcal{U}_{0}} \rightarrow M_{W_{0}}$ and $j_{\mathcal{W}_{1}}: M_{\mathcal{U}_{1}} \rightarrow M_{W_{1}}$ the associated elementary embeddings, we have  $M_{W_{0}}=M_{W_{1}}=M$.
\item $j_{\mathcal{W}_{0}}.j_{\mathcal{U}_{0}}=j_{\mathcal{W}_{1}}.j_{\mathcal{U}_{1}}$.
\end{enumerate}
\[ \begin{tikzcd}
V \arrow{r}{j_{\mathcal{U}_{0}}} \arrow[swap]{d}{j_{\mathcal{U}_{1}}} & M_{\mathcal{U}_{0}} \arrow{d}{j_{\mathcal{W}_{0}}} \\%
M_{\mathcal{U}_{1}} \arrow{r}{j_{\mathcal{W}_{1}}}& M
\end{tikzcd}
\]
}
\end{defn}

\begin{prop}{\em (c.f. [\cite{Apt2020}, \textbf{Proposition 1}]).}
{\em Assume UA. Let $\gamma = \vert \delta\vert$. If $\lambda$ is a measurable cardinal such that $o(\lambda) = \delta$, then
the number of normal measures $\lambda$ carries is $\gamma$.}
\end{prop}

\section{Removing the assumption that all strongly compact cardinals are limits of measurable cardinals} 
\subsection{Proving Theorem 1.1} We start with a sequence of strongly compact cardinals, which sequence has ordertype $\rho\in (\omega,\omega_{1}^{V}]$, as assumed in [\cite{Dim2011}, \textbf{chapter 2}, $\S$4] and construct our desired symmetric extension. Intuitively, we replace the injective tree Prikry forcing for type 1 cardinals as done in [\cite{Dim2011}, \textbf{chapter 2}, $\S$4] with the strongly compact Prikry forcing as in \cite{AH1991}.

{\bf{\underline{Defining the ground model ($V$)}:}}
Let $V$ be a model of ZFC where for some ordinal $\rho\in (\omega,\omega_{1}]$, there is a $\rho$-long sequence $\langle\kappa_{\epsilon}: 0<\epsilon<\rho\rangle$ of strongly compact cardinals. Let $\eta$ be the limit of this sequence. Let $Reg^{\eta}$ be the set of infinite regular cardinals $\alpha\in (\omega,\eta)$. We classify each $\alpha\in Reg^{\eta}$ in three types as follows.
\begin{itemize}
\item \textbf{(type 0).} If $\alpha\in(\omega,\kappa_{1})$.
\item \textbf{(type 1).} If $\alpha\geq\kappa_{1}$ and there is a largest $\kappa_{\epsilon}\leq\alpha$, i.e., $\alpha\in [\kappa_{\epsilon},\kappa_{\epsilon+1})$. 
\item \textbf{(type 2).} If $\alpha\geq\kappa_{1}$ and there is no largest stongly compact $\leq\alpha$, then let $\beta_{\alpha}=\cup\{\kappa_{\epsilon}:\kappa_{\epsilon}<\alpha\}$. We ditto Gitik's treatment for type 2 cardinals from chapter 2, section 4 of \cite{Dim2011}.
\end{itemize}

{\bf{\underline{Defining a symmetric inner model of a forcing extension of $V$}:}} Let $Reg^{\eta}_{0}$ be the set of all regular type 0 cardinals in $(\omega, \eta)$, $Reg^{\eta}_{1}$ be the set of all regular type 1 cardinals in $(\omega, \eta)$ and $Reg^{\eta}_{2}$ be the set of all regular type 2 cardinals in $(\omega, \eta)$. 

{\bf{\underline{Defining the partially ordered set}:}}
\begin{itemize}
    \item Let $\mathbb{P}_{\alpha}=\{p:\omega\rightharpoonup \alpha:\vert p \vert <\omega\}$ for every $\alpha\in Reg^{\eta}_{0}$ and $\mathbb{P}_{0}=\Pi^{fin}_{\alpha\in Reg^{\eta}_{0}}\mathbb{P}_{\alpha}$.
    \item Let $\mathcal{U}$ be the fine measure on $\mathcal{P}_{\kappa_{\epsilon}}(\kappa_{\epsilon+1})$, then we let $\mathbb{P}_{\kappa_{\epsilon}}$ to be the strongly compact Prikry forcing $\mathbb{P}_{\mathcal{U}}$. Let $\mathbb{P}_{1}=\Pi^{fin}_{0<\epsilon<\rho}\mathbb{P}_{\kappa_{\epsilon}}$ be the finite support product of $\mathbb{P}_{\kappa_{\epsilon}}$ where $0<\epsilon<\rho$.
    
    \item For each $\alpha\in Reg^{\eta}_{2}$, let $\mathbb{P}_{\alpha}$ be the forcing notion as described in [\cite{Dim2011}, \textbf{chapter 2}, $\S$4] for type 2 cardinals. Let $\mathbb{P}_{2}=\Pi^{fin}_{\alpha\in Reg^{\eta}_{2}}\mathbb{P}_{\alpha}$.
\end{itemize}
Let the desired forcing notion $\mathbb{P}$ be the product of $\mathbb{P}_{0}$, $\mathbb{P}_{1}$ and $\mathbb{P}_{2}$. Let $G$ be a $\mathbb{P}$-generic filter. 
  
{\bf{\underline{Defining the symmetric inner model}:}} 
We consider our symmetric inner model $\mathcal{N}$ to be the least model of ZF extending $V$ such that $V[G\restriction X]\subseteq \mathcal{N}$ for each $X\in \mathcal{I}$ where $\mathcal{I}$ is described as follows.
\begin{itemize}
    \item For every finite $e_{0}\subseteq Reg^{\eta}_{0}$, we define $E_{e_{0}}=\{p\restriction e_{0} : p\in \mathbb{P}_{0}\}$.
    \item For $m<\omega$ and $e_{1}=\{\alpha_1,...,\alpha_m\}\subseteq Reg^{\eta}_{1}$ a sequence of inaccessible cardinals in $Reg^{\eta}_{1}$ such that for each  $\alpha_{i}\in e_{1}$, there is a distinct $\epsilon_{\alpha_{i}}\in Ord$ such that $\alpha_{i}\in [\kappa_{\epsilon_{\alpha_{i}}},\kappa_{\epsilon_{\alpha_{i}}+1})$,\footnote{i.e., if $\alpha_{i}\not=\alpha_{j}\in e_{1}$, $\alpha_{i}\in [\kappa_{\epsilon_{\alpha_{i}}},\kappa_{\epsilon_{\alpha_{i}+1}})$ and $\alpha_{j}\in [\kappa_{\epsilon_{\alpha_{j}}},\kappa_{\epsilon_{\alpha_{j}}+1})$ then $\epsilon_{\alpha_{i}}\not=\epsilon_{\alpha_{j}}$.} we define $E_{e_{1}}=\Pi_{i\in \{1,2,...,m\}}\mathbb{P}_{\mathcal{U}_{\epsilon_{\alpha_{i}}}\restriction \alpha_{i}}$ 
    where $\mathcal{U}_{\epsilon_{\alpha_{i}}}\restriction \alpha_{i}$ is the fine measure on $\mathcal{P}_{\kappa_{\epsilon_{\alpha_{i}}}}(\alpha_{i})$ induced by some fine measure
    $\mathcal{U}_{\epsilon_{\alpha_{i}}}$ on
    $\mathcal{P}_{\kappa_{\epsilon_{\alpha_{i}}}}(\kappa_{\epsilon_{\alpha_{i}}+1})$ and
    $\mathbb{P}_{\mathcal{U}_{\epsilon_{\alpha_{i}}}\restriction \alpha_{i}}$ is the strongly compact Prikry forcing with respect to the fine measure
    $\mathcal{U}_{\epsilon_{\alpha_{i}}}\restriction \alpha_{i}$.
    \item For every finite $e_{2}\subseteq Reg^{\eta}_{2}$, we define $E_{e_{2}}=\{\overrightarrow{T}\restriction e_{2} : \overrightarrow{T}\in \mathbb{P}_{2}\}$.
    \item Let $\mathcal{I}=\{E_{e_{0}}\times E_{e_{1}}\times E_{e_{2}}:  e_{0}$ is any  finite subset of $Reg^{\eta}_{0}$, $e_{2}$ is any  finite subset of $Reg^{\eta}_{2}$ and $e_{1}$ is any finite collection of inaccessible cardinals in $Reg^{\eta}_{1}$ such that for each $\alpha_{i}\in e_{1}$, there is a distinct $\epsilon_{\alpha_{i}}\in Ord$ such that $\alpha_{i}\in [\kappa_{\epsilon_{\alpha_{i}}},\kappa_{\epsilon_{\alpha_{i}}+1}) \}$.
\end{itemize}
 
Formally, we define $\mathcal{N}$ as follows.
Let $\mathcal{L}$ be the forcing language with respect to $\mathbb{P}$. Let $\mathcal{L}_{1}\subseteq \mathcal{L}$ be a ramified sublanguage which contains symbols $\overrightarrow{v}$ for each $v\in V$, a predicate symbol $\overrightarrow{V}$ (to be interpreted as $\overrightarrow{V}(\overrightarrow{v})\leftrightarrow v\in V$), and symbols $\overline{G\restriction X}$ for each $X\in \mathcal{I}$. $\mathcal{N}$ is then defined inside $V[G]$ as follows.
\begin{itemize}
    \item $\mathcal{N}_{0}=\emptyset$.
    \item $\mathcal{N}_{\alpha+1}=\{x\subseteq \mathcal{N}_{\alpha}: x\in V[G]$ and is definable over $\langle\mathcal{N}_{\alpha}, \epsilon, c\rangle_{c\in \mathcal{N}_{\alpha}}$ by a formula $\phi\in\mathcal{L}_{1}$ of rank $\leq\alpha\}$.
    \item $\mathcal{N}_{\alpha}=\cup_{\beta<\alpha} \mathcal{N}_{\beta}$ for $\alpha$ a limit ordinal.
    \item $\mathcal{N}=\cup_{\alpha\in Ord} \mathcal{N}_{\alpha}$.
\end{itemize}

We recall the homogeneity of $\mathbb{P}_{0}$ (c.f. [\cite{Dim2011}, \textbf{chapter 1}, $\S$ 3]), the homogeneity of strongly compact Prikry forcing from [\cite{AH1991}, \textbf{Lemma 2.1}], homogeneity of injective tree-Prikry forcing from [\cite{Dim2011}, \textbf{Lemma 2.15}] and [\cite{Dim2011}, \textbf{Lemma 2.23}]. We also recall the fact that finite support product of weakly (cone) homogeneous forcing are weakly (cone) homogeneous. Consequently, we can obtain the desired homogeneity of $\mathbb{P}$ and observe the following lemma.

\begin{lem}
{\em If $X'$ is a set of ordinals in $\mathcal{N}$, then for some $X\in \mathcal{I}$, $X'\in V[G\restriction X]$.}
\end{lem}

We recall the Prikry like lemma for the injective tree Prikry forcing [\cite{Dim2011}, \textbf{Lemma 2.24}] and the Prikry like lemma for the strongly compact Prikry forcing [\cite{AH1991}, \textbf{Lemma 1.1}]. We apply this to show that all $\kappa_{\alpha}$ for $0<\alpha<\rho$, and their limits are still cardinals in $\mathcal{N}$. 

\begin{lem}{\em For every $0<\epsilon < \rho$, $\kappa_{\epsilon}$ is a cardinal in $\mathcal{N}$. Consequently, their limits are also preserved.}
\end{lem}
\begin{proof}
For the sake of contradiction we assume that for some $0<\epsilon<\rho$, there is some $\beta<\kappa_{\epsilon}$ and a bijection $f:\beta\rightarrow\kappa_{\epsilon}$ in $\mathcal{N}$. By \textbf{Lemma 3.1}, for some $X\in\mathcal{I}$, $f\in V[G\restriction X]$. 
Let $X$ be $E_{e_{0}}\times E_{e_{1}}\times E_{e_{2}}$ such that $e_{0}$ is some finite subset of $Reg^{\eta}_{0}$, $e_{2}$ is some finite subset of $Reg^{\eta}_{2}$ and $e_{1}$ is a finite collection of inaccessible cardinals in $Reg^{\eta}_{1}$ such that for each $\alpha_{i}\in e_{1}$, there is a distinct $\epsilon_{\alpha_{i}}\in Ord$ such that $\alpha_{i}\in [\kappa_{\epsilon_{\alpha_{i}}},\kappa_{\epsilon_{\alpha_{i}}+1})$. We may imagine $V[G\restriction X]$ as $V[G\restriction E_{e_{0}}][G\restriction E_{e_{1}}][G\restriction E_{e_{2}}]$ and show that $f$ is not added in $V[G\restriction E_{e_{0}}][G\restriction E_{e_{1}}][G\restriction E_{e_{2}}]$ to obtain a contradiction.

{\bf{\underline{(Step 1) $f$ is not added in $V[G\restriction E_{e_{0}}]$}:}} Clearly, $E_{e_{0}}$ is $\kappa_{\epsilon}$-c.c. So $f$ is not added in $V[G\restriction E_{e_{0}}]$.

{\bf{\underline{(Step 2) $f$ is not added in $V[G\restriction E_{e_{0}}][G\restriction E_{e_{1}}]$}:}}
Let $\{\alpha_{1},...,\alpha_{m}\}$ be an increasing enumeration of $e_{1}$, and let for each $1\leq i\leq m$ there is a distinct $\epsilon_{\alpha_{i}}$ such that $\alpha_{i}\in [\kappa_{\epsilon_{\alpha_{i}}},\kappa_{\epsilon_{\alpha_{i}}+1})$. Let $1\leq j\leq m$ be the greatest such that $\kappa_{\epsilon}>\alpha_{j}$. We can write $E_{e_{1}}$ as $\Pi_{i=1,...,j }\mathbb{P}_{\mathcal{U}_{\epsilon_{\alpha_{i}}}\restriction \alpha_{i}}\times \Pi_{i=j+1,...,m}\mathbb{P}_{\mathcal{U}_{\epsilon_{\alpha_{i}}}\restriction \alpha_{i}}$ where for each $1\leq i\leq m$, $\mathcal{U}_{\epsilon_{\alpha_{i}}}\restriction \alpha_{i}$ is the fine measure on $\mathcal{P}_{\kappa_{\epsilon_{\alpha_{i}}}}(\alpha_{i})$ and $\mathbb{P}_{\mathcal{U}_{\epsilon_{\alpha_{i}}}\restriction \alpha_{i}}$ is the strongly compact Prikry forcing with respect to the fine measure $\mathcal{U}_{\epsilon_{\alpha_{i}}}\restriction \alpha_{i}$.
Clearly, $\Pi_{i=j+1,...,m}\mathbb{P}_{\mathcal{U}_{\epsilon_{\alpha_{i}}}\restriction \alpha_{i}}$ do not add any bounded subset of $\kappa_{\epsilon}$ following the Prikry like lemma  for the strongly compact Prikry forcing.
Moreover for each $i=1,...,j$, $\mathbb{P}_{\mathcal{U}_{\epsilon_{\alpha_{i}}}\restriction \alpha_{i}}$ is $(\alpha_{i}^{<\kappa_{\epsilon_{\alpha_{i}}}})^{+}-c.c.$
and since $\kappa_{\epsilon}$ is inaccessible, the $\kappa_{\epsilon}$-c.c. So, the partial order $\Pi_{i=1,...,j}\mathbb{P}_{\mathcal{U}_{\epsilon_{\alpha_{i}}}\restriction \alpha_{i}}$ has $\kappa_{\epsilon}$-c.c. Thus, $f$ is not added in $V[G\restriction E_{e_{0}}][G\restriction E_{e_{1}}]$ either.

{\bf{\underline{(Step 3) $f$ is not added in $V[G\restriction E_{e_{0}}][G\restriction E_{e_{1}}][G\restriction E_{e_{2}}]$}:}}
Clearly, $E_{e_{2}}=E_{e_{2}\cap\kappa_{\epsilon}}\times E_{e_{2}\backslash\kappa_{\epsilon}}$ where $E_{e_{2}\cap\kappa_{\epsilon}}$ is $\kappa_{\epsilon}$-c.c. and $E_{e_{2}\backslash\kappa_{\epsilon}}$ does not add bounded subsets to $\kappa_{\epsilon}$ following the Prikry like lemma for the injective tree Prikry forcing from [\cite{Dim2011}, \textbf{Lemma 2.24}]. Thus, no such $f$ can exist in $V[G\restriction E_{e_{0}}][G\restriction E_{e_{1}}][G\restriction E_{e_{2}}]$ also.
\end{proof}

\begin{lem}
{\em In $\mathcal{N}$, the regular cardinals of type 2 have collapsed to their singular limits of strongly compact cardinals below them and if $\alpha\in (\kappa_{\epsilon},\kappa_{\epsilon+1})$ is a regular cardinal of type 1 where $0<\epsilon<\rho$, then $(\vert\alpha\vert=\kappa_{\epsilon})^{\mathcal{N}}$.}
\end{lem}

\begin{proof}
Following [\cite{Dim2011}, \textbf{Lemma 2.28}] the regular cardinals of type 2 have collapsed to their singular limits of strongly compact cardinals below them. 
Following [\cite{AH1991}, \textbf{Lemma 2.4}] if $\alpha\in (\kappa_{\epsilon},\kappa_{\epsilon+1})$ is a regular cardinal of type 1 where $0<\epsilon<\rho$, then $(\vert\alpha\vert=\kappa_{\epsilon})^{\mathcal{N}}$.\footnote{The argument goes as follows. Let $\alpha\in (\kappa_{\epsilon},\kappa_{\epsilon+1})$ is a type 1 regular cardinal and $\beta\in (\alpha,\kappa_{\epsilon+1})$ be an inaccessible cardinal in $V$. We first show that $\alpha$ is no longer a cardinal in $V[G\restriction E_{\{\beta\}}]$. 
More specifically, we show that there are no cardinals in the interval $(\kappa_{\epsilon},\beta]$ in $V[G\restriction E_{\{\beta\}}]$. For the sake of contrary, let $\alpha_{1}\in (\kappa_{\epsilon},\beta]$ be the least cardinal in $V$ which remains a cardinal in $V[G\restriction E_{\{\beta\}}]$. We observe contradiction in each of the  following two cases.

\textbf{Case (i). If $\alpha_{1}$ is a regular cardinal in $V$.}
We can see that $cf(\alpha_{1})=\omega$ in $V[G\restriction E_{\{\beta\}}]$. By the leastness of the cardinality of $\alpha_{1}$, $\alpha_{1}=\kappa_{\epsilon}^{+}$. But, $cf(\kappa_{\epsilon}^{+})=\omega$ in $V[G\restriction E_{\{\beta\}}]$ is impossible since $V[G\restriction E_{\{\beta\}}]$ is a model of $AC$.

\textbf{Case (ii). \textbf{If $\alpha_{1}$ is a singular cardinal in $V$}.}
Once more $\alpha_{1}=\kappa_{\epsilon}^{+}$ in $V[G\restriction E_{\{\beta\}}]$ which is impossible since $V[G\restriction E_{\{\beta\}}]$ is a model of $AC$, and so the successor cardinal cannot be a singular cardinal. 

Thus, there are no cardinals in the interval $(\kappa_{\epsilon},\beta]$ in $V[G\restriction E_{\{\beta\}}]$. As $V[G\restriction E_{\{\beta\}}]\subseteq \mathcal{N}$, the collapsing function for $\alpha$ is in $\mathcal{N}$ as well. Consequently, $\alpha$ is not a cardinal in $\mathcal{N}$ and so $(\vert\alpha\vert=\kappa_{\epsilon})^{\mathcal{N}}$.}
\end{proof}

Consequently, we can have the following corollary similar to [\cite{Dim2011}, \textbf{Corollary 2.29}]. 
\begin{corr}{\em In $\mathcal{N}$, a cardinal in $(\omega,\eta)$ is a successor cardinal if and only if it is in $\{\kappa_{\epsilon}: \epsilon<\rho\}$ and a cardinal in $(\omega,\eta)$ is a limit cardinal if and only if it is a limit in the sequence $\{\kappa_{\epsilon}: \epsilon<\rho\}$ in $V$.}
\end{corr}

\begin{lem}{\em In $\mathcal{N}$, every ordinal in $Reg^{\eta}$ is singular of cofinality $\omega$. Consequently, the interval $(\omega,\eta)$ only contains singular cardinals of cofinality $\omega$}.
\end{lem}

\begin{proof}Let $\alpha$ is in $Reg^{\eta}$ is either of type 1 or type 2. There is a $\omega$-Prikry sequence cofinal in $\alpha$ supported by $\{\alpha\}$ following  [\cite{AH1991}, \textbf{Lemma 2.3}] and  [\cite{Dim2011}, \textbf{Lemma 2.26}].
\end{proof}

\begin{lem}{\em In $\mathcal{N}$, all cardinals in $(\omega,\eta)$ carry a Rowbottom filter.}\end{lem}

\begin{proof}
Let $\kappa$ be a limit cardinal in $\mathcal{N}$. First we prove that if $V[G\restriction X]$ is an intermediate model between $V$ and $\mathcal{N}$ for some $X\in \mathcal{I}$, then we can write $V[G\restriction X]$ as $V[G_{1}][G_{2}]$ where $G_{1}$ is $V$-generic over a forcing notion $\mathbb{P}_{1}$ such that $\vert \mathbb{P}_{1} \vert<\kappa$ and $G_{2}$ is $V[G_{1}]$-generic over a forcing notion $\mathbb{P}_{2}$ such that $\mathbb{P}_{2}$ does not add bounded subsets to $\kappa$. We may imagine $V[G\restriction X]$ as $V[G\restriction E_{e_{0}}][G\restriction E_{e_{1}}][G\restriction E_{e_{2}}]$ where $e_{0}$ is some finite subset of $Reg^{\eta}_{0}$, $e_{2}$ is some finite subset of $Reg^{\eta}_{2}$ and $e_{1}$ is a finite collection of inaccessible cardinals in $Reg^{\eta}_{1}$ such that for each $\alpha_{i}\in e_{1}$, there is a distinct $\epsilon_{\alpha_{i}}\in Ord$ such that $\alpha_{i}\in [\kappa_{\epsilon_{\alpha_{i}}},\kappa_{\epsilon_{\alpha_{i}}+1})$.

{\bf{\underline{Step 1}:}} Clearly, $\vert E_{e_{0}}\vert<\kappa$. 

{\bf{\underline{Step 2}:}}
Let $\{\alpha_{1},...,\alpha_{m}\}$ be an increasing enumeration of $e_{1}$, and let for each $1\leq i\leq m$ there is a distinct $\epsilon_{\alpha_{i}}$ such that $\alpha_{i}\in [\kappa_{\epsilon_{\alpha_{i}}},\kappa_{\epsilon_{\alpha_{i}}+1})$. Let $1\leq j\leq m$ be the greatest such that $\kappa>\alpha_{j}$. We can write $E_{e_{1}}$ as $\Pi_{i=1,...,j }\mathbb{P}_{\mathcal{U}_{\epsilon_{\alpha_{i}}}\restriction \alpha_{i}}\times \Pi_{i=j+1,...,m}\mathbb{P}_{\mathcal{U}_{\epsilon_{\alpha_{i}}}\restriction \alpha_{i}}$ where for each $1\leq i\leq m$, $\mathcal{U}_{\epsilon_{\alpha_{i}}}\restriction \alpha_{i}$ is the fine measure on $\mathcal{P}_{\kappa_{\epsilon_{\alpha_{i}}}}(\alpha_{i})$ and $\mathbb{P}_{\mathcal{U}_{\epsilon_{\alpha_{i}}}\restriction \alpha_{i}}$ is the strongly compact Prikry forcing with respect to the fine measure $\mathcal{U}_{\epsilon_{\alpha_{i}}}\restriction \alpha_{i}$.
Clearly, $\Pi_{i=j+1,...,m}\mathbb{P}_{\mathcal{U}_{\epsilon_{\alpha_{i}}}\restriction \alpha_{i}}$ do not add any bounded subset of $\kappa$ following the Prikry like lemma  for the strongly compact Prikry forcing.
Moreover $\vert \Pi_{i=1,...,j}\mathbb{P}_{\mathcal{U}_{\epsilon_{\alpha_{i}}}\restriction \alpha_{i}}\vert<\kappa$ since for each $i=1,...,j$, $\vert \mathbb{P}_{\mathcal{U}_{\epsilon_{\alpha_{i}}}\restriction \alpha_{i}}\vert<\kappa_{\epsilon_{\alpha_{i}}+1}$ (c.f. [\cite{AH1991}, \textbf{Lemma 2.5}]).

{\bf{\underline{Step 3}:}} Following [\cite{ADK2016}, \textbf{Lemma 2.31}] (or more appropriately the works done in [\cite{Dim2011}, \textbf{$\S$ 4}, \textbf{chapter 2}], $E_{e_{2}}$ can be written as $E_{e_{2}\cap\kappa}\times E_{e_{2}\backslash\kappa}$ where $\vert E_{e_{2}\cap\kappa}\vert<\kappa$ and $E_{e_{2}\backslash\kappa}$ does not add bounded subsets to $\kappa$ following the Prikry like lemma for the injective tree Prikry forcing from [\cite{Dim2011}, \textbf{Lemma 2.24}]. 

{\bf{\underline{All limit cardinals in $(\omega,\eta)$ carry a Rowbottom filter in $\mathcal{N}$}:}} By [\cite{Dim2011}, \textbf{Lemma 2.31}] (or more appropriately the arguments in [\cite{ADK2016}, \textbf{Lemma 2.5}]) and [\cite{Kan2003}, \textbf{Theorem 8.7}], we can see that all limit cardinals in $(\omega,\eta)$ are Rowbottom cardinals carrying a Rowbottom filter. For reader's convenience, we write down the proof. 

{\bf{\underline{Step 1}:}} Let $\kappa$ be a limit cardinal of $(\omega, \eta)$ in $\mathcal{N}$. We first prove that there is some $X_{0} \in \mathcal{I}$ such that $V[G\restriction X_{0}] \models ``cf(\kappa) = \omega$ and $\kappa = sup(\chi_{i} : i < \omega)$, where each $\chi_{i}$ is measurable". By \textbf{Corollary 3.4}, $\kappa$ is a limit in the sequence $\{\kappa_{\epsilon}:\epsilon <\rho\}$ in $V$.
Also by \textbf{Lemma 3.5}, the interval $(\omega,\eta)$ only contains singular cardinals of cofinality $\omega$ in $\mathcal{N}$. So, $cf(\kappa)=\omega$ in $\mathcal{N}$. By \textbf{Lemma 3.1}, there is a $X_{0}
\in \mathcal{I}$ such that $V[G\restriction X_{0}] \models ``cf(\kappa) = \omega$ and $\kappa =sup(\chi_{i}:i <\omega)$, where each $\chi_{i}$ is measurable in $V$". 
However, as proved before, we can write $V[G\restriction X_{0}] = V[H_{0}][H_{1}]$, where $H_{0}$ is $V$-generic over a partial ordering $\mathbb{P}_{1}$ such that $\vert \mathbb{P}_{1}\vert < \kappa$, and $H_{1}$ is $V[H_{0}]$-generic over a partial ordering $\mathbb{Q}$ which adds no bounded subsets of $\kappa$. By the results of \cite{LS1967}, it is still the case that $V[H_{0}] \models ``\kappa$ is a limit of measurable cardinals''. Because forcing with $\mathbb{Q}$ adds no bounded subsets
of $\kappa$, $V[H_{0}][H_{1}] = V[G\restriction X_{0}]$ as desired. For the rest of the proof, we fix $X_{{0}}$, $\{\chi_{i}: i < \omega\}\in V[G\restriction X_{0}]$, and $\{\mu_{i}: i < \omega\} \in V[G\restriction X_{0}]$ such that $V[G\restriction X_{0}]\models ``\kappa = sup(\chi_{i}:i < \omega)$, where each $\chi_{i}$ is a measurable cardinal, and each $\mu_{i}$ is a normal measure over $\chi_{i}$”. We also define $\mathcal{F}\in V[G\restriction X_{0}]$ as follows.
\begin{equation}
    \mathcal{F} = \{A \subseteq \kappa : \exists n < \omega \forall i \geq n [A \cap \chi_{i} \in \mu_{i}]\}
\end{equation}
Clearly, $\mathcal{F}$ generates a filter (in any model of ZF in which it is a
member).

{\bf{\underline{Step 2}:}} We show that $\kappa$ is a Rowbottom cardinal carrying a Rowbottom filter in $\mathcal{N}$. Let $\gamma<\kappa$ be arbitrary and $f:[\kappa]^{<\omega}\rightarrow\gamma$ be a partition function. Since $f$ can be coded as a subset of $\kappa$ by \textbf{Lemma 3.1}, there is a $X\in \mathcal{I}$ such that $f\in V[G\restriction X]$. Without loss of generality, by coding if necessary, we may assume in addition that $V[G\restriction X]\supseteq V[G\restriction X_{0}]$. As before, we have seen that if $V[G\restriction X]$ is an intermediate model between $V$ and $\mathcal{N}$ for some $X\in \mathcal{I}$, then we can write $V[G\restriction X]$ as $V[G_{1}][G_{2}]$ where $G_{1}$ is $V$-generic over a forcing notion $\mathbb{P}_{1}$ such that $\vert \mathbb{P}_{1} \vert<\kappa$ and $G_{2}$ is $V[G_{1}]$-generic over a forcing notion $\mathbb{P}_{2}$ such that $\mathbb{P}_{2}$ does not add bounded subsets to $\kappa$. 
Therefore, by the results of \cite{LS1967}, we may further assume that in $V[G\restriction X]$, a final segment $\mathbb{F}$ of $\langle\chi_{i}:i< \omega\rangle$ is composed of measurable
cardinals, and that for any $i$ such that $\chi_{i} \in \mathbb{F}$, $\mu'_{i}$ defined in $V[G\restriction X]$ by 
\begin{equation}
    \mu'_{i} = \{A \subseteq \chi_{i}:\exists Y \in \mu_{i} [Y \subseteq A]\}
\end{equation}
is a normal measure over $\chi_{i}$. Let $n_{0}$ be least such that $\chi_{n_{0}}\in \mathbb{F}$. In $V[G\restriction X]$ define the following set.
\begin{equation}
    \mathcal{F}^{*} = \{A\subseteq \kappa: \exists n \geq n_{0} \forall i \geq n[A \cap \chi_{i} \in \mu'_{i}]\} 
\end{equation}

By [\cite{Kan2003}, \textbf{Theorem 8.7}] for some $Z^{*}\in \mathcal{F}^{*}$, $Z^{*}$ is homogeneous
for $f$. By the definitions of $\mathcal{F}$ and $\mathcal{F}^{*}$ and the fact that every $\mu'_{i}$ measure 1 set contains a $\mu_{i}$ measure 1 set for $\chi_{i}\in \mathcal{F}$, it then immediately follows that for some $Z\in\mathcal{F}$, $Z\subseteq Z^{*}$, $Z$ is homogeneous for $f$. Thus, $F\in V[G\restriction X] \subseteq \mathcal{N}$ generates a Rowbottom filter for $\kappa$ in $\mathcal{N}$. 


{\bf{\underline{All successor cardinals in $(\omega,\eta)$ carry a Rowbottom filter in $\mathcal{N}$}:}} Adopting the {\em appropriate automorphism technique} from [\cite{AH1991}, \textbf{Lemma 3.1}] we observe that in $\mathcal{N}$, all the successor cardinals in $(\omega,\eta)$ can carry Rowbottom filter. In $\mathcal{N}$, if a cardinal $\kappa$ in $(\omega,\eta)$ is a successor cardinal, then there is an $\epsilon<\rho$ such that $\kappa=\kappa_{\epsilon}$. We show that $\kappa_{\epsilon}$ carries a Rowbottom filter in $V[G\restriction E_{\{\kappa_{\epsilon}\}}]\subset \mathcal{N}$. 
Firstly, we see that $\kappa_{\epsilon}$ carries a Rowbottom filter in $V[G']$ where $G'$ is a $V$-generic filter over $\mathbb{P}_{\kappa_{\epsilon}}$. Suppose for the sake of contradiction $p=\langle p_{0},...,p_{r},u\rangle\in G'$ forces that $F:[\mathcal{P}_{\kappa_{\epsilon}}(\kappa_{\epsilon+1})]^{<\omega}\rightarrow \gamma<\kappa_{\epsilon}$ is a counter example to the Rowbottomness of $\kappa_{\epsilon}$. 
Let $\mathcal{U}$ be the fine measure on $\mathcal{P}_{\kappa_{\epsilon}}(\kappa_{\epsilon+1})$ such that $\mathbb{P}_{\kappa_{\epsilon}}=\mathcal{P}_{\mathcal{U}}$.

{\bf{\underline{(Step 1) Defining $\mathcal{U}_{\kappa_{\epsilon}}$ and $\mathcal{F}_{\kappa_{\epsilon}}$}:}} 
Let $k:\mathcal{P}_{\kappa_{\epsilon}}(\kappa_{\epsilon+1})\rightarrow \kappa_{\epsilon}$ be a map. We define $\mathcal{U}_{\kappa_{\epsilon}}$ to be the push-forward ultrafilter $k_{*}(\mathcal{U})$. We may assume that $\mathcal{U}_{\kappa_{\epsilon}}$ is a normal measure on $\kappa_{\epsilon}$. Otherwise, we can change $\mathcal{U}_{\kappa_{\epsilon}}$ so that it becomes normal, as follows. 
\begin{itemize}
        \item Let $r:\mathcal{P}_{\kappa_{\epsilon}}(\kappa_{\epsilon+1})\rightarrow \kappa_{\epsilon}$ be the least function in the ultrapower of $\mathcal{P}_{\kappa_{\epsilon}}(\kappa_{\epsilon+1})$ such that $r$ is not a constant function on a set in $\mathcal{U}$, but $r(p)<k(p)$ on a set in $\mathcal{U}$.  
        \item Define a map $l:\mathcal{P}_{\kappa_{\epsilon}}(\kappa_{\epsilon+1})\rightarrow \mathcal{P}_{\kappa_{\epsilon}}(\kappa_{\epsilon+1})$ as follows.
        
        \begin{center}
            $l(p)=(p\backslash \kappa_{\epsilon})\cup (p\cap r(p))$.
        \end{center}
        
        We can see that $l_{*}(\mathcal{U})$ is a fine measure on $\mathcal{P}_{\kappa_{\epsilon}}(\kappa_{\epsilon+1})$ and $k_{*}(l_{*}(\mathcal{U}))$ is a normal measure on $\kappa_{\epsilon}$. 
\end{itemize}
    
Let, $\mathcal{F}_{\kappa_{\epsilon}}=\{f:f:[\mathcal{P}_{\kappa_{\epsilon}}(\kappa_{\epsilon+1})]^{<\omega}\rightarrow\mathcal{U}_{\kappa_{\epsilon}}\}$. 
    
{\bf{\underline{(Step 2) Defining a subset $X_{H,f}$ of $\kappa_{\epsilon}$ and $\mathcal{U}_{H}$}:}}
For any $f\in \mathcal{F}_{\kappa_{\epsilon}}$ and any $H$ which is $V$-generic over $\mathbb{P}_{\mathcal{U}}$ we define the following subset $X_{H,f}$ of $\kappa_{\epsilon}$ in $V[H]$.
\begin{center}
    $X_{H,f}=[f(\emptyset)\cap (p_{0}\cap \kappa_{\epsilon})]\cup [f(p_{0})\cap [(p_{1}\cap \kappa_{\epsilon})\backslash(p_{0}\cap \kappa_{\epsilon})]]\cup [f(p_{0},p_{1})\cap [(p_{2}\cap \kappa_{\epsilon})\backslash(p_{1}\cap \kappa_{\epsilon})]]\cup..$.
\end{center}
We define, $\mathcal{U}_{H}=\{X_{H,f}:{f\in \mathcal{F}_{\kappa_{\epsilon}}}\}$.

We can clearly observe that $\mathcal{U}_{H}$ is a filter on $\kappa_{\epsilon}$. 
We recall $\mathcal{F}=\{f: f$ is a function from $[\mathcal{P}_{\kappa_{\epsilon}}(\kappa_{\epsilon+1})]^{<\omega}$ to $\mathcal{U}\}$ from the definition of $\mathbb{P}_{\mathcal{U}}$.
Let $\mathcal{T}$ be the collection of finite sequences of 0's and 1's. 

{\bf{\underline{(Step 3) Defining an appropriate pair}:}}
For $\pi\in \mathcal{T}$, $g\in \mathcal{F}$, $h\in \mathcal{F}_{\kappa_{\epsilon}}$, $\sigma=\langle s_{0},...,s_{k}\rangle\in [\mathcal{P}_{\kappa_{\epsilon}}(\kappa_{\epsilon+1})]^{<\omega}$ with each $s_{l}\cap \kappa_{\epsilon}$ a cardinal for $0\leq l\leq k$, $\tau=\langle t_{0},...,t_{n}\rangle\in [\kappa_{\epsilon}]^{<\omega}$, we say $\langle \sigma,\tau\rangle$ is {\em appropriate} for $\pi$, $g$, $h$ if and only if the following holds.
\begin{itemize}
    \item $s_{0}\cap \kappa_{\epsilon} < s_{1}\cap \kappa_{\epsilon}<...$,
    \item $t_{0}<t_{1}...$,
    \item $t_{i}\not= s_{j}\cap \kappa_{\epsilon}$ for all $i$ and $j$.
    \item In case $\{t_{i}\}_{i<n}$, $\{s_{j}\cap \kappa_{\epsilon}\}_{j<k}$ are arranged in order, we have a sequence $\rho$ with the following.
    \begin{itemize}
        \item len($\rho$) = len ($\pi$).
        \item If $\pi(i)=0$ then $\rho(i)=t_{j}$ for some $j$ and $\rho(i)\in h(t_{0},...,t_{j-1})$.
        \item If $\pi(i)=1$ then $\rho(i)\not \in \tau$ and $\rho(i)\in g(t_{0},...,t_{j})$ where $t_{j}$ is the greatest member of $\tau$ below $\rho(i)$.
    \end{itemize}
\end{itemize}

Similar to the claim in the proof of [\cite{AH1991}, \textbf{Lemma 3.1}] we can observe the following.

\begin{claim}
{\em for all $\pi\in\mathcal{T}$ and for all $\sigma\in [\mathcal{P}_{\kappa_{\epsilon}}(\kappa_{\epsilon+1})]^{<\omega}$ extending $\langle p_{0},...p_{r}\rangle$ there 
are $g\in\mathcal{F}$, $h\in \mathcal{F}_{\kappa_{\epsilon}}$, $\alpha<\kappa_{\epsilon}$ such that for all $\langle \sigma',\tau\rangle$ 
appropriate for $\pi, g, h$, $\langle \sigma\frown\sigma',g\rangle\Vdash ``F(\tau)=\alpha"$.}
\end{claim}

Now let $\sigma$ be $\langle p_{0},...,p_{r}\rangle$, and choose $g_{\pi},h_{\pi},\alpha_{\pi}$ for each $\pi\in \mathcal{T}$. Consider the following.

\begin{itemize}
    \item $g$ be the intersection $(\cap g_{\pi})\cap f$,
    \item $h$ be $\cap h_{\pi}$,
    \item $Z=\{\alpha_{\pi}\}_{\pi\in\mathcal{T}}$.
\end{itemize}

Let $H$ be a $V$-generic filter over $\mathbb{P}_{\mathcal{U}}$ such that $\langle \sigma,g\rangle\in H$. For any $\tau\in [X_{H,h}]^{<\omega}$, we can find $\sigma'$ and a $\pi$ such that $\langle \sigma', \tau\rangle$ is appropriate for $g_{\pi}$, $h_{\pi}$ and $\pi$. Thus $\langle \sigma\frown\sigma',g_{\pi}\rangle\Vdash ``F(\tau)=\alpha_{\pi}"$ and so $\langle \sigma\frown\sigma',g\rangle\Vdash ``F(\tau)\in Z"$ and $\langle \sigma,g\rangle\Vdash ``F''[X_{H,h}]^{<\omega}\subseteq Z$''. Now $\vert Z \vert\leq\omega$ contradicts the assumption that $\langle \sigma,f\rangle$ forces that $F$ is a counterexample to Rowbottomness of $\kappa_{\epsilon}$. Consequently, we can observe that $\mathcal{U}_{G'}$ is a Rowbottom filter on $\kappa_{\epsilon}$ in $V[G']$. 

Now, the definition of $X_{H,f}$ above for a $V$-generic filter $H$ over $\mathbb{P}_{\kappa_{\epsilon}}$, depends only on $H\restriction \kappa_{\epsilon}$. Consequently, $\mathcal{U}_{G'}$ is in $V[G\restriction E_{\{\kappa_{\epsilon}\}}]$. 
\end{proof}

Arguments from [\cite{Dim2011}, \textbf{Lemma 2.20}, \textbf{Lemma 2.30}] guarantees that all cardinals in the interval $(\omega,\eta)$ are almost Ramsey. 

\begin{question}{\em (asked in \cite{ADK2016}).}
{\em Is it possible to remove the additional assumption that `every strongly compact cardinals are the limit of measurable cardinals' from} [\cite{ADK2016}, \textbf{Theorem 1.1}]?
\end{question}
\section{Proving Theorem 1.2}


{\bf{\underline{Defining the ground model ($V$)}:}}
Let $V$ be a model of ZFC where $f:\omega\rightarrow 2$ be an arbitrary given function. Let $\langle\kappa_n: 1\leq n<\omega\rangle$ be a sequence of strongly compact cardinals, $\eta$ be the limit of the sequence $\langle\kappa_n: 1\leq n<\omega\rangle$ and $Reg^{\eta}$ be the set of regular cardinals in $(\omega,\eta)$.
Let $Reg^{\eta}_{0}$ be the set of all regular cardinals in $(\omega, \kappa_{1})$. We define the following sets, $Reg^{\eta}_{1}=\{\alpha\in Reg^{\eta} : \exists n\in\omega, \alpha\in [\kappa_{n+1},\kappa_{n+2})$ and $f(n)=0 \}$ and $Reg^{\eta}_{2}=\{\alpha\in Reg^{\eta} : \exists n\in\omega, \alpha\in [\kappa_{n+1},\kappa_{n+2})$ and $f(n)=1\}$. 

{\bf{\underline{Defining a symmetric inner model of the forcing extension of $V$}:}}

{\bf{\underline{Defining the partially ordered set}:}}
 \begin{itemize}
    \item  Let $\mathbb{P}_{\alpha}=\{p:\omega\rightharpoonup \alpha:\vert p \vert <\omega\}$ for every $\alpha\in Reg^{\eta}_{0}$ and $\mathbb{P}_{0}=\Pi^{fin}_{\alpha\in Reg^{\eta}_{0}}\mathbb{P}_{\alpha}$.
    
    \item Given $n<\omega$, if $f(n)=0$
    let $\mathcal{U}$ be the fine measure on $\mathcal{P}_{\kappa_{n+1}}(\kappa_{n+2})$. We let $\mathbb{P}_{\kappa_{n+1}}$ to be the strongly compact Prikry forcing $\mathbb{P}_{\mathcal{U}}$. 
    
    Let $\mathbb{P}_{1}=\Pi^{fin}_{n<\omega,f(n)=0}\mathbb{P}_{\kappa_{n+1}}$ be the finite support product of $\mathbb{P}_{\kappa_{n+1}}$ when $f(n)=0$.
    
    \item Given $n<\omega$, if $f(n)=1$ and $\alpha\in [\kappa_{n+1},\kappa_{n+2})$ then we let $\mathbb{P}_{\alpha}=\{p:\kappa_{n+1}\rightharpoonup \alpha:\vert p \vert <\kappa_{n+1}\}$. Let $\mathbb{P}_{2}=\Pi^{fin}_{\alpha\in Reg^{\eta}_{2}}\mathbb{P}_{\alpha}$.
\end{itemize}

Let the desired forcing notion $\mathbb{P}$ be the product of $\mathbb{P}_{0}$, $\mathbb{P}_{1}$ and $\mathbb{P}_{2}$. Let $G$ be $V$-generic over $\mathbb{P}$.

{\bf{\underline{Defining the symmetric inner model}:}}
We consider our symmetric inner model $\mathcal{N}$ to be the least model of ZF extending $V$ such that $V[G\restriction X]\subseteq \mathcal{N}$ for each $X\in \mathcal{I}$ where $\mathcal{I}$ is described as follows.
\begin{itemize}
\item For every finite $e_{0}\subseteq Reg^{\eta}_{0}$, we define $E_{e_{0}}=\{p\restriction e_{o} : p\in \mathbb{P}_{0}\}$.

    \item 
    For $m<\omega$ and $e_{1}=\{\alpha_1,...,\alpha_m\}\subseteq Reg^{\eta}_{1}$ a sequence of inaccessible cardinals in $Reg^{\eta}_{1}$ such that for each  $\alpha_{i}\in e_{1}$, there is a distinct $\epsilon_{\alpha_{i}}\in Ord$ such that $\alpha_{i}\in [\kappa_{\epsilon_{\alpha_{i}}},\kappa_{\epsilon_{\alpha_{i}}+1})$,\footnote{i.e., if $\alpha_{i}\not=\alpha_{j}\in e_{1}$, $\alpha_{i}\in [\kappa_{\epsilon_{\alpha_{i}}},\kappa_{\epsilon_{\alpha_{i}+1}})$ and $\alpha_{j}\in [\kappa_{\epsilon_{\alpha_{j}}},\kappa_{\epsilon_{\alpha_{j}}+1})$ then $\epsilon_{\alpha_{i}}\not=\epsilon_{\alpha_{j}}$.} we define $E_{e_{1}}=\Pi_{i\in \{1,...,m\}}\mathbb{P}_{\mathcal{U}_{\epsilon_{\alpha_{i}}}\restriction \alpha_{i}}$ 
    where $\mathcal{U}_{\epsilon_{\alpha_{i}}}\restriction \alpha_{i}$ is the fine measure on $\mathcal{P}_{\kappa_{\epsilon_{\alpha_{i}}}}(\alpha_{i})$ induced by some fine measure
    $\mathcal{U}_{\epsilon_{\alpha_{i}}}$ on
    $\mathcal{P}_{\kappa_{\epsilon_{\alpha_{i}}}}(\kappa_{\epsilon_{\alpha_{i}}+1})$ and
    $\mathbb{P}_{\mathcal{U}_{\epsilon_{\alpha_{i}}}\restriction \alpha_{i}}$ is the strongly compact Prikry forcing with respect to the fine measure
    $\mathcal{U}_{\epsilon_{\alpha_{i}}}\restriction \alpha_{i}$.
    
    \item 
    For every finite $e_{2}\subseteq Reg^{\eta}_{2}$, we define $E_{e_{2}}=\{p\restriction e_{2} : p\in \mathbb{P}_{2}\}$.

    \item Let $\mathcal{I}=\{E_{e_{0}}\times E_{e_{1}}\times E_{e_{2}}:  e_{0}$ is any  finite subset of $Reg^{\eta}_{0}$, $e_{2}$ is any  finite subset of $Reg^{\eta}_{2}$ and $e_{1}$ is any finite collection of inaccessible cardinals in $Reg^{\eta}_{1}$ such that for each $\alpha_{i}\in e_{1}$, there is a distinct $\epsilon_{\alpha_{i}}\in Ord$ such that $\alpha_{i}\in [\kappa_{\epsilon_{\alpha_{i}}},\kappa_{\epsilon_{\alpha_{i}}+1}) \}$.
    
\end{itemize}
Formally, we define $\mathcal{N}$ as follows.
Let $\mathcal{L}$ be the forcing language with respect to $\mathbb{P}$. Let $\mathcal{L}_{1}\subseteq \mathcal{L}$ be a ramified sublanguage which contains symbols $\overrightarrow{v}$ for each $v\in V$, a predicate symbol $\overrightarrow{V}$ (to be interpreted as $\overrightarrow{V}(\overrightarrow{v})\leftrightarrow v\in V$),  and symbols $\overline{G\restriction X}$ for each $X\in \mathcal{I}$. $\mathcal{N}$ is then defined in $V[G]$ as follows.
\begin{itemize}
    \item $\mathcal{N}_{0}=\emptyset$.
    \item $\mathcal{N}_{\alpha+1}=\{x\subseteq \mathcal{N}_{\alpha}: x\in V[G]$ and is definable over $\langle\mathcal{N}_{\alpha}, \epsilon, c\rangle_{c\in \mathcal{N}_{\alpha}}$ by a formula $\phi\in\mathcal{L}_{1}$ of rank $\leq\alpha\}$.
    \item $\mathcal{N}_{\alpha}=\cup_{\beta<\alpha} \mathcal{N}_{\beta}$ for $\alpha$ a limit ordinal.
    \item $\mathcal{N}=\cup_{\alpha\in Ord} \mathcal{N}_{\alpha}$.
\end{itemize}

We recall the homogeneity of [\cite{Dim2011}, \textbf{Chapter 1}, \textbf{$\S$3}], the homogeneity of strongly compact Prikry forcing from [\cite{AH1991}, \textbf{Lemma 2.1}] and the fact that finite support product of weakly (cone) homogeneous forcing notions are weakly (cone) homogeneous. Consequently, we can obtain the desired homogeneity of $\mathbb{P}$ and observe the following lemma.

\begin{lem}{\em If $X'$ is a set of ordinals in $\mathcal{N}$, then there is some $X\in \mathcal{I}$, such that $X'\in V[G\restriction X]$.}\end{lem}

Similar to \textbf{Lemma 3.2} we can see that for every $0<n<\omega$, $\kappa_{n}$ is a cardinal in $\mathcal{N}$.
Similar to \textbf{Lemma 3.3} (more appropriately following [\cite{AH1991}, \textbf{Lemma 2.4}]), we can see that for any $0<n<\omega$ such that $f(n)=0$ if $\alpha\in (\kappa_{n+1},\kappa_{n+2})$ then $\alpha$ has collapsed to $\kappa_{n+1}$ in $\mathcal{N}$. 

\begin{lem}
{\em For any $0<n<\omega$ such that $f(n)=1$ if $\alpha\in (\kappa_{n+1},\kappa_{n+2})$ then $\alpha$ has collapsed to $\kappa_{n+1}$ in $\mathcal{N}$. Moreover, $\kappa_{n+1}$ do not carry any uniform ultrafilter in $\mathcal{N}$.}
\end{lem}

\begin{proof}
Fix an $0<n<\omega$ such that $f(n)=1$ and $\alpha\in (\kappa_{n+1},\kappa_{n+2})$. Since $\alpha\in Reg^{\eta}_{2}$ and by definition of $\mathcal{N}$, $G\restriction E_{\alpha}$ is in $\mathcal{N}$ we can see that $\alpha$ collapses to $\kappa_{n}$ in $\mathcal{N}$. By [\cite{KH2019}, \textbf{Theorem 2.4}], $\kappa_{n+1}$ do not carry any uniform ultrafilter.
\end{proof}

Following [\cite{AH1991}, \textbf{Lemma 2.3}], we can see that if $f(n)=0$ then $\kappa_{n+1}$ becomes a singular cardinal of cofinality $\omega$ in $\mathcal{N}$. Adopting the {\em appropriate automorphism technique} from [\cite{AH1991}, \textbf{Lemma 3.1}] as done in \textbf{Lemma 3.6}, we observe that in $\mathcal{N}$, all the singular cardinals in $(\omega,\eta)$ can carry Rowbottom filter as well. Arguments of [\cite{Dim2011}, \textbf{Lemma 2.20}] guarantees that all the singular cardinals in the interval $(\omega,\eta)$ are almost Ramsey. 
\section{Weakening the assumption of supercompactness by strong compactness}
\subsection{Proving Observation 1.3} We reduce the large cardinal assumption of [\cite{AC2013}, \textbf{Theorem 1}].

{\bf{\underline{Defining the ground model ($V$)}:}}
We start with a model $V_{0}$ of ZFC where $\kappa$ is a strongly compact cardinal, $\theta$ an ordinal and GCH holds. By [\cite{ADU2019}, \textbf{Theorem 3.1}] we can obtain a forcing extension $V$ where $2^{\kappa}=\theta$ and strong compactness of $\kappa$ is preserved. We assume $\lambda>\kappa$ in $V$ such that $(cf(\lambda))^{V}<\kappa$.

{\bf{\underline{Defining a symmetric inner model of a forcing extension of $V$}:}}

{\bf{\underline{Defining the partially ordered set}:}} Let $\mathcal{U}$ be a fine measure on $\mathcal{P}_{\kappa}(\lambda)$ and $\mathbb{P}=\mathbb{P}_{\mathcal{U}}$ be the strongly compact Prikry forcing. Let $G$ be $V$-generic over $\mathbb{P}_{\mathcal{U}}$.

{\bf{\underline{Defining the symmetric inner model}:}} We consider the model  constructed in [\cite{AH1991}, $\S$2]. In particular, we consider our symmetric inner model $\mathcal{N}$ to be the least model of ZF extending $V$ and containing $r\restriction\delta$ for each inaccessible $\delta\in[\kappa,\lambda)$ where $r\restriction{\delta}=\{<p_{0}\cap \delta,...p_{n}\cap\delta>: \exists f\in \mathcal{F} [\langle p_{0},...p_{n},f\rangle\in G]\}$ but not the $\lambda$-sequence of $r\restriction\delta$'s. 

Formally, we define $\mathcal{N}$ as follows.
Let $\mathcal{L}$ be the forcing language with respect to $\mathbb{P}$. Let $\mathcal{L}_{1}\subseteq \mathcal{L}$ be a ramified sublanguage which contains symbols $\overrightarrow{v}$ for each $v\in V$, a predicate symbol $\overrightarrow{V}$ (to be interpreted as $\overrightarrow{V}(\overrightarrow{v})\leftrightarrow v\in V$), and symbols $\overline{r\restriction \delta}$ for each inaccessible $\delta\in [\kappa,\lambda)$. $\mathcal{N}$ is then defined inside $V[G]$ as follows.
\begin{itemize}
    \item $\mathcal{N}_{0}=\emptyset$.
    \item $\mathcal{N}_{\alpha+1}=\{x\subseteq \mathcal{N}_{\alpha}: x\in V[G]$ and $x$ is definable over $\langle\mathcal{N}_{\alpha}, \epsilon, c\rangle_{c\in \mathcal{N}_{\alpha}}$ by a formula $\phi\in\mathcal{L}_{1}$ of rank $\leq\alpha\}$.
    \item $\mathcal{N}_{\alpha}=\cup_{\beta<\alpha} \mathcal{N}_{\beta}$ for $\alpha$ a limit ordinal.
    \item $\mathcal{N}=\cup_{\alpha\in Ord} \mathcal{N}_{\alpha}$.
\end{itemize}

We follow the homogeneity of strongly compact Prikry forcing mentioned in [\cite{AH1991}, \textbf{Lemma 2.1}] to observe the following lemma.

\begin{lem}
{\em If A$\in \mathcal{N}$ is a set of ordinals, then $A\in V[r\restriction \delta]$ for some inaccessible $\delta\in[\kappa,\lambda)$.}
\end{lem}

\begin{lem}
{\em In $\mathcal{N}$, $\kappa$ is a strong limit cardinal that is a limit of inaccessible cardinals.}
\end{lem}

\begin{proof}
Since $V\subseteq \mathcal{N}\subseteq V[G]$ and $\mathbb{P}$ does not add bounded subsets to $\kappa$, $V$ and $\mathcal{N}$ have same bounded subsets of $\kappa$.\footnote{We can also follow \textbf{Lemma 2.2} of \cite{AH1991}.} Consequently, in $\mathcal{N}$, $\kappa$ is a limit of inaccessible cardinals and thus a strong limit cardinal as well.
\end{proof}

\begin{lem}
{\em If $\gamma\geq\lambda$ is a cardinal in $V$, then $\gamma$ remains a cardinal in $\mathcal{N}$.}
\end{lem}

\begin{proof}
For the sake of contradiction, let $\gamma$ is not a cardinal in $\mathcal{N}$. Then there is a bijection $f:\alpha\rightarrow\gamma$ for some $\alpha<\gamma$ in $\mathcal{N}$. Since $f$ can be coded by a set of ordinals, by \textbf{Lemma 5.1} $f\in V[r\restriction\delta]$ for some inaccessible $\delta\in [\kappa,\lambda)$. Since $GCH$ is assumed in $V_{0}$ we have $(\delta^{<\kappa})^{V_{0}}=\delta$, and since $Add(\kappa,\theta)$ preserves cardinals and adds no sequences of ordinals of length less than $\kappa$, we conclude that $(\delta^{<\kappa})^{V}=(\delta^{<\kappa})^{V_{0}}=\delta$.     
Now $\mathbb{P}_{\mathcal{U}\restriction \delta}$ is $(\delta^{<\kappa})^{+}$-c.c. in $V$ and hence $\delta^{+}$-c.c. in $V$. Consequently, $\gamma$ is a cardinal in $V[r\restriction\delta]$ which is a contradiction.
\end{proof}

\begin{lem}
{\em In $\mathcal{N}$, $cf(\kappa)=\omega$. Moreover, $(\kappa^{+})^{\mathcal{N}}=\lambda$ and $cf(\lambda)^{\mathcal{N}} = cf(\lambda)^V$.}
\end{lem}

\begin{proof}
For each $\delta\in [\kappa,\lambda)$, we have $V[r\restriction \delta]\subseteq \mathcal{N}$. Consequently, $cf(\kappa)^{\mathcal{N}}=\omega$ since $cf(\kappa)^{V[r\restriction \kappa]}=\omega$.
Following [\cite{AH1991}, \textbf{Lemma 2.4}], every ordinal in $(\kappa,\lambda)$ which is a cardinal in $V$ collapses to have size $\kappa$ in $\mathcal{N}$, and so $(\kappa^{+})^{\mathcal{N}}=\lambda$. 
Since $V$ and $\mathcal{N}$ have same bounded subsets of $\kappa$, we see that $cf(\lambda)^{\mathcal{N}}=cf(\lambda)^{V}<\kappa$. 
\end{proof}

We can see that since, $V\subseteq \mathcal{N}$ and $(2^{\kappa}=\theta)^{V}$, there is a $\theta$-sequence of distinct subsets of $\kappa$ in $\mathcal{N}$. Since $cf(\kappa^{+})^{\mathcal{N}}<\kappa$ we can also see that $AC_{\kappa}$ fails in $\mathcal{N}$.
\section{(A Remark): Strong compactness and number of normal measures a successor of singular cardinal can carry}

\subsection{Proving Remark 1.4} 
We observe that in a symmetric inner model based on strongly compact Prikry forcing, the successor of a singular cardinal of cofinality $\omega$ can carry arbitrary (regular cardinal) number of normal measures under certain large cardinal assumptions. 

{\bf{\underline{Defining the ground model ($V$)}:}} Let $V$ be a model of ZFC + GCH+ UA. In $V$, let $\kappa<\lambda$ are such that $\kappa$ is strongly compact and $\lambda$ is the least measurable cardinal above $\kappa$ such that $o(\lambda)=\delta$ for some ordinal $\delta\leq \lambda^{++}$. Since $V$ models ``UA + $o(\lambda) = \delta$”, by \textbf{Proposition 2.9}, the number of normal measures $\lambda$ carries in $V$ is $\vert \delta\vert$. 

We consider the symmetric inner model $\mathcal{N}$ construction from \textbf{Theorem 1.3} (or more appropriately from [\cite{AH1991}, $\S$2]), based on strongly compact Prikry forcing. Consequently, $\lambda=\kappa^{+}$ and $cf(\kappa)=\omega$ in $\mathcal{N}$. We recall that the following facts hold in $\mathcal{N}$.

\begin{enumerate}
    \item (\textbf{Lemma 5.1}). If $A\in \mathcal{N}$ is a set of ordinals, then $A\in V[r\restriction \delta]$ for some inaccessible $\delta\in[\kappa,\lambda)$. 
    \item Any intermediate extension $V[r\restriction\delta]\subseteq \mathcal{N}$ is a small forcing extension of $V$ with respect to $\lambda$.\footnote{By a {\em small forcing extension with respect to $\kappa$} we mean a forcing extension $V[G]$ obtained from $V$ after forcing with a partially ordered set of size less than $\kappa$.} (c.f. [\cite{AH1991}, \textbf{Lemma 2.5}]).
\end{enumerate}
By (1) and (2) we can observe the following.
\begin{itemize}
    \item For any normal measure $\mathcal{U}$ over $\lambda$ in $V$, the set
    $\mathcal{U}_{0} = \{x \subseteq \lambda : \exists y \subseteq x[y \in \mathcal{U}]\}$ is a normal measure over $\lambda$ in $\mathcal{N}$ by [\cite{Apt2001}, \textbf{Lemma 2.4}]. 
    \item If $\mathcal{U}^{*}\in \mathcal{N}$ is a normal measure over $\lambda$, then for some normal measure $\mathcal{U}\in V$ over $\lambda$, $\mathcal{U}^{*} = \{x \subseteq \lambda : \exists y \subseteq x [y \in \mathcal{U}]\}$ by  [\cite{Apt2001}, \textbf{Lemma 2.5}].
\end{itemize}
Thus $\lambda$ remains a measurable cardinal with $\vert \delta\vert^{\mathcal{N}}$ many normal measures in $\mathcal{N}$. Moreover, if $\gamma\geq\lambda$ is a cardinal in $V$, then $\gamma$ remains a cardinal in $\mathcal{N}$ (c.f. \textbf{Lemma 5.3}).

\end{document}